\documentclass[twoside, 12pt]{amsart}
\usepackage[norelsize]{algorithm2e}
\usepackage{etex}
\usepackage{amsfonts,amssymb,amsmath}
\usepackage[english]{babel}
\usepackage{layout}
\usepackage{pstricks}
\usepackage{pst-plot}
\usepackage{geometry}
\usepackage{setspace}
\usepackage{appendix}
\usepackage{wrapfig}
\usepackage{blindtext}
\usepackage{tabularx}
\usepackage{subfig}
\usepackage[hidelinks]{hyperref}
\usepackage{latexsym}
\usepackage{a4} 
\usepackage{latexsym} 
\usepackage{exscale} 
\hypersetup{backref, colorlinks=true, linkcolor=blue, linktoc=all, linktocpage, linkbordercolor = red, citecolor = violet, citebordercolor = green, pdfborder = {0 0 1 [1 0] }}
\usepackage{cite}
\usepackage[all,cmtip]{xy}
\usepackage[all]{xy}
\newtheorem{thm}{Theorem}[section]

\newtheorem{lem}[thm]{Lemma}
\newtheorem{defn}[thm]{Definition}
\newtheorem{prop}[thm]{Proposition}
\newtheorem{coro}[thm]{Corollary}

\newtheorem*{conjA}{Conjecture A}
\newtheorem*{conjB}{Conjecture B}
\theoremstyle{remark}
\newtheorem{remark}[thm]{Remark}

\theoremstyle{definition}
\newtheorem{exa}[thm]{Example}

\def\Z{\mathbb{Z}}
\def\N{\mathbb{N}}

\def\bb{\mathbb}
\def\cal{\mathcal}
\def\ol{\overline}

\def\rw{\rightarrow}

\def\Rw{\Rightarrow}

\def\Lr{\Leftrightarrow}

\def\Q{\mathbb{Q}}
\def\a{\alpha}

\def\beq{\begin{eqnarray*}}
\def\enq{\end{eqnarray*}}

\def\gm{\mathfrak{m}}

\def\stack{\stackrel}

\def\co{\cal{O}}

\def\rm{\mathrm}

\begin{document}
\title{Log ramification via curves in rank 1} 
\author{Ivan Barrientos}
\address{Audible, Data Sciences, 1 Washington Pl, Newark, NJ 07102}
\email{corps.des.nombres@gmail.com}
\subjclass[2000]{}
\begin{abstract}
We prove that in rank 1, the Abbes-Saito log conductor is determined by its restriction to curves. This result is essentially established by analyzing Artin-Schreier-Witt extensions. Consequently, we confirm an expectation of H. Esnault and M. Kerz. We also conjecture that this result holds in arbitrary finite rank. As an application, we translate recent results in higher class field theory of M. Kerz and S. Saito to the log-ramification context.
\end{abstract}
\maketitle
\tableofcontents
\section{Introduction}
Ramification plays a central role in understanding how structures behave under extensions. For example, in number theory one studies the behavior of primes under Galois extensions and distinguishes between tame (well-behaved) and wild (more complicated) ramification. The analogous concept in differential equations is that of extending the space of solutions by studying the asymptotic behavior of solutions near singular points. In this case, tame ramification corresponds to studying regular singular points and wild ramification to studying irregular singular points. Following P. Deligne, one further translates the differential equations point of view to the world of algebraic geometry by studying integrable connections along the boundary of a curve.

It was realized that one could essentially study regular singular points of differential equations from a purely algebraic approach. As a specific example, Y. Andr\'e recently provided an \emph{algebraic} proof that an integrable connection is regular if and only if it is regular under restriction to curves, \cite{And07}. The spirit of this approach is that ramification of $n$-dimensional structures is detected by $1$-dimensional substructures.

The goal of this work is to similarly study the relationship between ramification and restriction to curves in the context of wild ramification for varieties (in characteristic $p>0$). The measure of wild ramification that we use is a conductor. We prove that for abelian extensions of varieties, this conductor is computed by restricting to subcurves of the given variety. Our approach is entirely algebraic.

Let us now proceed to a more formal introduction, starting with Galois extensions of discrete valuation rings.
Let $R$ be a discrete valuation ring with residue field $\kappa$ and fraction field $K$. Suppose $L$ is a finite Galois extension of $K$. If $\kappa$ is perfect, then there is a satisfactory theory of higher ramification groups and conductors attached to $L/K$, \cite{ser1} and \cite{lau81}. On the other hand, if $\kappa$ is not perfect complications may arise; for example, the ring of integers of $L$ may not be a monogenic extension of $R$\footnote{For a comprehensive survey on ramification see \cite{xiaozhukov}.}.
In the 1980's, J.-L. Brylinski \cite{bry83} and K. Kato \cite{Kat89}. defined conductors when $\kappa$ is not necessarily perfect and $L/K$ is attached to a character of rank 1.
Recently, Abbes and Saito succeeded in providing a general definition of higher ramification groups in the context of $\ell$-adic sheaves of finite rank that agrees with the classical case of a perfect residue field in \cite{abbsai02} and \cite{abbsai11}. Consequently, they also define a conductor which agrees with Brylinski-Kato's in rank 1, see \cite{abbsai09}.

In another direction, in the 1970's P. Deligne initiated a program of measuring ramification of sheaves along a divisor in terms of transversal curves, see \cite{del76}, which was further developed by G. Laumon in \cite{lau81} and \cite{lau82}. One of the principal aims of these works was to achieve new Euler-Poincar\'e formulas (over algebraically closed fields).
A natural question was then whether one could follow P. Deligne's program and express Abbes-Saito's conductor in terms of curves. In this direction, S. Matsuda established results in the so-called `non-log' case of Brylinski-Kato's conductor in \cite{mat} (see also \cite[Coro 2.8]{kersai}). There are also stronger results for the `non-log' case recently obtained by T. Saito in \cite{sai13}.

The main contribution of this article is an algebraic proof that Brylinski-Kato's `log' conductor is calculated by restriction to curves, see Theorem \ref{tangentrank1}. In other words, we establish that Abbes-Saito's `log' conductor in rank 1 is expressed in terms of curves. Our approach is a modification of P. Deligne's original idea of analyzing ramification via transversal curves - we instead consider \emph{all} curves on a scheme and pay special attention to \emph{tangent} curves. We conjecture that our result also holds for $\ell$-adic sheaves of finite rank in Conjectures A and B below. We hope our conjectures can be used to further the analogy between irregular singular points and wild ramification in higher dimensional class field theory.

As applications of our main result, we confirm an expectation of H. Esnault and M. Kerz \cite[\S3, pg.\,11]{esnker12} in the rank 1 smooth case in Theorem \ref{thmexp2}. We also obtain a reformulation of the recent `non-log' Existence Theorem of higher class field theory obtained by M. Kerz and S. Saito, \cite[Coro IV]{kersai} in the `log' case, see \S\ref{modulus}.

The work presented here forms part of the author's PhD thesis.
\section{Conjectures A and B}
Fix a perfect field $k$ of characteristic $p>0$.
Let $X/k$ be a normal variety and $j: U\hookrightarrow X$ a $k$-smooth open subvariety such that the closed complement $E:=X\backslash U $ is the support of an effective Cartier divisor. Let $D\in \rm{Div}^+(X)$ be an effective Cartier divisor with $\rm{Supp}(D)\subset E$. Let $\cal{F}$ be a smooth $\ell$-adic sheaf of finite rank on
$U$.
Denote by $\rm{Sw}_{.}(-)_{\rm{log}}$ the Abbes-Saito log conductor of a constructible sheaf at a codimension 1 point of a scheme, \cite[Definition 8.10(i)]{abbsai11}. We omit the subscript `log' when the scheme has dimension one. Note that this definition simplifies in rank 1, see Def. \ref{cond} and Prop. \ref{abbsaicond} below.
\begin{defn}\label{cuU} For a scheme $X$, we let $\rm{Cu}(X)$ denote the set of normalizations of closed integral
1-dimensional subschemes of $X$.
\end{defn}
Given an effective Cartier divisor $W$ on a scheme and a point $y\in \rm{Supp}(W)$, we denote the multiplicity of $y$ in $W$ by $m_y(W)$; it is a non-negative integer. The set of closed points of a scheme $Y$ is denoted by $|Y|$.
\subsection{Statement of Conjecture A}\label{conjB}
Given $C\in \rm{Cu}(U)$, let $\bar{C}$ denote the normalization of the closure of $C$ in $X$. Moreover, let $\phi: C\rw U$ and $\bar{\phi}:\bar{C}\rw X$ denote the induced morphisms.
\begin{defn} (\cite[Def. 3.6]{esnker12}) Given $C\in \rm{Cu}(U)$, set
$$\rm{Sw}(\phi^*\cal{F}):=\sum_{z\in |\bar{C}|}
\rm{Sw}_z(\bar{\phi}^*j_!\cal{F}).[z] \in \rm{Div}^+(\bar{C}).$$ We say that \emph{the ramification of $\cal{F}$ is bounded by
$D\in \rm{Div}^+(X)$} if for each $C\in \rm{Cu}(U)$ the following inequality of divisors in
$\rm{Div}^+(\bar{C})$ holds: $$\rm{Sw}(\phi^*\cal{F})\leq \bar{\phi}^*(D).$$
\end{defn} In this case, we formally write $$\rm{Sw}(\cal{F})\leq D.$$
\begin{conjA}\label{expectation2} For $D$ a simple normal crossings divisor (sncd) with $\rm{Supp}(D)\subset E$, the following are equivalent (cf. \cite{esnker12}).
\begin{itemize}
\item[(i)] $\rm{Sw}(\cal{F})\leq D$.
\item[(ii)] For each open immersion $j':U\subset X'$ over $k$ such
that $X'\backslash U$ is a simple normal crossing divisor and for any morphism $h:
X'\rw X,$ \emph{that extends} $j:U\rw X$, $$\rm{Sw}_{\xi
'}(h^*j'_!\cal{F})_{\log}\leq m_{\xi '}(h^*(D)),$$ for each generic
point $\xi '\in D': = X'\backslash U$.
\end{itemize}
\end{conjA}
\subsection{Statement of Conjecture B}\label{conjA} \begin{conjB}\label{expectation1} Suppose that $D = E$ is an irreducible divisor with generic point $\xi$. Then
\begin{equation}\label{exp1limit} \sup_{\bar{\phi}}
\frac{\rm{Sw}_z(\bar{\phi}^*j_!\cal{F})}{m_z(\bar\phi^*D)} = \rm{Sw}_\xi
(j_!\cal{F})_\rm{log},\end{equation}
where the supremum ranges over all $\bar{C}\in
\rm{Cu}(X)$ and over all closed points $z\in|\bar{C}|$ with multiplicity $m_z(\bar\phi^*D)>0.$
\end{conjB}
\begin{remark}\label{sumirreducible} If $E$ is not necessarily irreducible and $D$ is an sncd, then the left-hand side of \eqref{exp1limit} should be replaced by a sum over the components of $D$. Moreover, we expect this formula to hold at each smooth point of $D$. \end{remark}
\begin{remark} In \cite[Remark 2.5.3]{zhu08}, a limit similar to that in \eqref{exp1limit} is mentioned for a smooth surface defined over an algebraically closed field.\end{remark}
\;
\;
Our main result is that if $X/k$ is smooth, $D$ is a simple normal crossing divisor on $X$, and $\cal{F}$ has rank 1, then Conjecture B is
true (see Theorem \ref{tangentrank1}) and implies Conjecture A (Theorem \ref{thmexp2}). We do not require resolution of singularities. On the other hand, if one insists that we begin with a smooth scheme $U/k$ and a \emph{proper} scheme $X/k$ that contains $U$ as an open dense subscheme, then resolution of singularities seems to be required to apply our techniques.
We also note that in the case with $D=0$ (and arbitrary rank), the equivalence in Conjecture A is shown by M. Kerz and A. Schmidt in \cite{kerscht} relying on work of Wiesend in \cite{wie1}. Moreover, in the `non-log' rank 1 setting Conjecture A follows by work of Matsuda on Kato's Artin conductor, see \cite{mat} and \cite[Coro 2.8]{kersai}. We note that \cite{mat} focuses on smooth curves on $X$ that are transversal to components of $D$.
\subsection{Remarks on the higher rank case} We make some brief remarks relating recent work by T. Saito to both of the above conjectures for general rank. First, the inequality $\leq$ in Conjecture B (a form of `semi-continuity') is evidenced by \cite[Lemma 2.31]{sai09}. Furthermore, the implication (ii) $\Rw$ (i) in Conjecture A is also evidenced by \cite[Lemma 2.22]{sai09}. Finally, results from \cite{sai13} seem to shed light on the converse of these statements as well. In particular, a careful analysis of the refined Swan conductor seems to be indispensable and is part of future work.
\section{Background}
The Brylinski-Kato filtration is introduced in \S \ref{ASWdiscussion}. In Def. \ref{newschool} we define a conductor for finite extensions of discrete valuation fields with not-necessarily separable residue field extensions. In the next section, we provide an example that gives evidence for computing the log conductor in terms of a limit over tangent curves, see \S \ref{surfacecurves}. We also define the log conductor for rank 1 characters in \S \ref{swanvariety}.
\subsection{Conductor for general discrete valuation fields}\label{ASWdiscussion}
We mostly follow
the notation in \cite[\S 9]{abbsai09}. Let $(R, v)$ be a discrete
valuation ring of characteristic $p>0$ with normalized valuation $v$. Let $K$ be the fraction field of $R$ and fix a
separable closure $K^{\text{sep}}$ of $K$.
For an integer $n\geq 1,$ let $W_n(K)$ denote the truncated ring of Witt vectors of length $n$ over $K$. The Brylinski-Kato filtration on $W_n(K)$ is an increasing filtration defined as follows.
\begin{defn} \label{aswfil} For $m\geq 0$ define the subgroup \emph{$\rm{fil}_m W_n(K)$} of $W_n(K)$ as
$$\rm{fil}_m
W_n(K) = \left\{ (x_0, \ldots, x_{n-1}):p^{n-i-1} v(x_i)\geq-m\, , \; (0\leq i\leq n-1) \right\}.$$ \end{defn}
We have that for $\bold{x} = (x_0,\ldots, x_{n-1})\in W_n(K)$, then $\bold{x}\in \text{fil}_m W_n(K)$ if
\begin{equation} \label{filpiece}
m \geq \max\{-p^{n-1} v(x_0), \ldots, -v(x_{n - 1}), 0\}.
\end{equation}
This filtration yields a conductor for characters.
Set $$H^1(K) := \varinjlim_{n\geq 1} H^1(K, \Z/n) = H^1(K,
\Q/\Z).$$
\begin{remark}\label{cyclic} Note that since $\rm{Gal}(K^{\text{sep}}/K)$
is a compact topological group (under the usual Krull topology) and
$\Q/\Z$ is a discrete topological group, then an element $$\chi\in H^1(K) = \rm{Hom}_{\text{cont}}(\rm{Gal}(K^{\text{sep}}/K), \Q/\Z),$$
must have finite image in $\Q/\Z$ and this image is therefore
a finite cyclic group.\end{remark}
Put $\eta = \rm{Spec}(K)$ and recall that for $n\geq 1$ there is a short
exact sequence of sheaves on $\eta_{\text{\'et}}$ :
\begin{equation} \label{ASW} 0\longrightarrow \Z/p^n \longrightarrow W_n
\stack{F-1}{\longrightarrow} W_n \longrightarrow 0.\end{equation}
Now, the short exact sequence (\ref{ASW}) induces a surjective connecting homomorphism of groups
$$\delta_n: W_n(K)\rw H^1(K, \Z/p^n),$$ with $\rm{ker}(\delta_n) = (F-1)W_n(K)$ and we have a canonical isomorphism of groups
\begin{equation}\label{ASWK} W_n(K)/(F-1) \cong H^1(K,
\Z/p^n).\end{equation}
The isomorphism \eqref{ASWK} classifies finite
cyclic extensions of $K$ of order $p^{n'}$, for $0\leq n'\leq n$. Moreover, there is a commutative diagram
\begin{equation}\label{Vercommute}\xymatrix{
W_n(K) \ar[d]^{V} \ar[rr]^{\delta_n} && H^1(K, \Z /p^n) \ar[d]^{\cdot p}\\
W_{n+1}(K) \ar[rr]^{\delta_{n+1}} && H^1(K, \Z/p^{n+1}).}
\end{equation}
Now write
\begin{equation}\label{connecting} \text{fil}_m H^1(K,
\Z/p^n) := \delta_n (\text{fil}_m W_n(K)).\end{equation}
The Brylinski-Kato filtration induces a filtration on all of $H^1(K)$ via
\begin{equation} \text{fil}_m H^1(K):= H^1(K)[p'] + \text{fil}_m
H^1(K)[p^\infty],\label{primezero}\end{equation}
where
$$ H^1(K)[p'] = \varinjlim_{\ell\geq 1 ,(\ell,p)=1} H^1(K, \Z /\ell),$$ and
$$H^1(K)[p^\infty] = \varinjlim_{r\geq 1} H^1(K, \Z/p^r).$$
\begin{defn}\label{cond} Let $\chi\in H^1(K)$. The Swan \emph{conductor} of $\chi$,
denoted by $\rm{Sw}(\chi)_\rm{log}$, is the minimal integer $m\geq 0$ for which
$\chi\in\rm{fil}_m H^1(K)$.\end{defn}
\begin{remark}\label{newschool} If the residue field of $K$ is finite, this definition agrees with the classical definition by a theorem J.-L. Brylinski \cite[Corollary to Theorem 1]{bry83}. In a more general case, it is due to K. Kato \cite{Kat89}.\end{remark}
We also record the following important result of A. Abbes and T. Saito.
The Abbes-Saito log conductor is defined in \cite[Def. 8.10(i)]{abbsai11} and that this definition agrees with ours is given by \cite[Coro 9.12 and Def. 10.16]{abbsai09}. We record this fact as the following
\begin{prop}\label{abbsaicond} The conductor defined above in Def. \ref{cond} is equal to the log conductor of Abbes and Saito in rank 1.
\end{prop}

\begin{remark}\label{tame} If $\chi'\in H^1(K)[p']$, then
$\rm{Sw}(\chi') = 0$, by (\ref{primezero}) above. Hence, our study of wild ramification is reduced to that of
characters in $H^1(K)[p^\infty].$
\end{remark}

\subsection{Witt vectors and extensions} \label{wittandextensions}
\begin{lem}\label{maintain} Let $\chi\in H^1(K, \Z/p^n)$ and suppose $0\leq n'\leq n$. Then $\chi$ corresponds to a cyclic $\Z/p^{n'}$-extension of $K$ if and only its corresponding Witt vector is in the image of the iterated Verschiebung modulo $(F-1)$, $$V^{n-n'} : W_{n'}(K)/(F-1) \rw W_{n}(K)/(F-1).$$\end{lem}
\begin{proof}
This follows from the
isomorphism (\ref{ASWK}) and the commutative diagram (\ref{Vercommute}).
\end{proof}
Using the isomorphism (\ref{ASWK}), we can make explicit the $p$-cyclic
extensions of $K$ as follows. Suppose $n\geq 1$ and $L/K$ is Galois with
$\rm{Gal}(L/K) = \Z/p^{n}$. Let $\chi\in H^1(K, \Z/p^n)$ correspond to $L/K$. Then there are Witt vectors $$(x_0, \ldots,
x_{n-1})\in W_{n}(K), \; \;(\a_0, \ldots, \a_{n-1})\in
W_{n}(K^{\text{sep}}),$$ such that \begin{equation}\label{wittextension} L = K[\a_0, \ldots, \a_{n-1}],\end{equation}
where $(\a_0,\ldots, \a_{n-1})\in W_{n}(K^{\text{sep}})$ satisfies
$$(F-1)(\a_0,\ldots, \a_{n-1}) = (x_0, \ldots, x_{n-1}).$$ 
We may illustrate the calculation of the conductor of $L/K$ by choosing a `minimal lift' of $\bar{\bold{x}}$ as follows. Choose an element $\bold{y} = (y_0, \ldots, y_{n-1}) \in W_{n}(K)$ such that
\begin{itemize}
\item[(i)] $\bold{y} \equiv \bold{x} \mod (F-1)W_{n+1}(K)$,
\item[(ii)] there is a $k\in [0,n-1]$ such that $-p^{n-k}v(y_k) = \rm{Sw}(\chi).$
\end{itemize}
Clearly such a $\bold{y}$ exists simply because $W_{n+1}(K)\rw W_{n+1}(K)/(F-1)$ is surjective.
Fix a $\bold{y}$ satisfying (i) and (ii) above. By \eqref{filpiece},
\begin{equation}\label{computeswan}
\rm{Sw}(\chi) = \max\{-p^{n-1}.v(y_0), \ldots, -v(y_{n-1}), 0\}.
\end{equation}
\begin{exa}(Conductor of an Artin-Schreier extension) \label{addexample} Keep the above notation. If $n=1$, then $\rm{Sw}(\chi) = \sup\{-v(y_{0}), 0\}$, which is zero if and only if $L/K$ is tame.\begin{flushright} $\square$ \end{flushright}
\end{exa}
The following lemma further illustrates the computation of conductors from certain Witt vectors and will be used in Lemma \ref{limitfilmod} below.

\begin{lem}\label{gamma} Let $k$ be a perfect field of $\rm{char}(k) = p>0$, $K=k((t))$, $v_t$ the normalized discrete valuation on $K$, and $n\geq 0$. Define $\mu:  W_{n+1}(K) \rw \N$ by $$\mu((x_0, \ldots, x_n)) = \max\{0, -v_t(x_n), \ldots, -p^n v_t(x_0)\}.$$
Suppose $\mathbf{y} = (y_0, \ldots, y_n)\in W_{n+1}(K)$ satisfies 
\begin{itemize}
\item[(i)] $-v_t(y_n) > -p^{n-k} v_t(y_k)$, for $k=0, \ldots, n-1$ and $-v_t(y_n) >0$,
\item[(ii)] for each $z\in K, v_t(y_n + (F-1)z)\leq v_t(y_n)$.
\end{itemize}

Let $\chi\in H^1(K, \Q/\Z)$ correspond to the image of $\mathbf{y}$ in $W_{n+1}(K)/(F-1)$. Then $\rm{Sw}(\chi) = -v_t(y_n)$.

\end{lem}
Note that condition (i) is stronger than $\mu(\mathbf{y}) = -v_t(y_n)$.
\begin{proof} Let $N = -v_t(y_n)$. By definition of $\rm{Sw}(\chi)$, we have $N\geq \rm{Sw}(\chi)$. For an arbitrary $\mathbf{z}\in W_{n+1}(K)$ we need to show $\mu(\mathbf{y} + (F-1)\mathbf{z}) \geq N.$ If $\mu(\mathbf{z}) = 0$, Witt vector addition directly yields this inequality; so assume $\mu(\mathbf{z}) > 0$. Let $m$ be the minimal integer such that $-p^{n-m}v_t(z_m) = \mu(\mathbf{z})$.

First, suppose $\mu(\mathbf{z}) >N/p$.  By Witt vector addition, $\mu(\mathbf{y} + (F-1)\mathbf{z}) = \mu(F\mathbf{z})$.  This can be seen by looking at the valuation of the $m$-th place in both expressions. Hence, $\mu(\mathbf{y} + (F-1)\mathbf{z}) \geq N.$ Similarly, Witt vector addition shows that if $\mu(\mathbf{z}) < N/p$, then $\mu(\mathbf{y} + (F-1)\mathbf{z}) \geq N$.

Finally, suppose $\mu(\mathbf{z}) = N/p$. If $m<n$, we argue as in the first case above. If $m = n$,  we have $v_t(z_n) = -N/p$. But then $v_t(y + z_n^{-N}  - z_n) = -N/p > -N = v_t(y_n)$, which contradicts condition (ii).
\end{proof}

\section{Conductors on varieties}\label{swanvariety}
We mostly follow the notation in \cite{kersai}. Let $k$ be a perfect field and let $X/k$ be a normal variety. Let $U\subset X$ be an open subvariety that is smooth over $k$ and such that the reduced closed complement $E:=X\backslash U$ is the support of an effective Cartier divisor on $X$. The discrete valuation rings of interest to us will be those defined by points `at infinity' and those from closed points of certain curves on $X$. 
Denote by $I$ the set of generic points of $E$ and $X^{(1)}$ the set of codimension 1 points of $X$. Then $I\subset X^{(1)}$. For $\lambda\in I$, let $K_\lambda$ be the fraction field of the henselian discrete valuation ring $\co_{X,\lambda}^h$.
Recall that in Def. \ref{cond} we defined the Swan conductor for characters over discrete valuation fields.


Let $H^1(U, \Q/\Z)$ denote the group of continuous characters $\chi: \pi_1^{\rm{ab}}(U) \rw \Q/\Z$, where $\pi_1^{\rm{ab}}(U)$ is the abelianized \'{e}tale fundamental group of $U$.

\begin{defn}\label{condgen} Fix $\chi\in H^1(U,\Q/\Z)$. We define
$$\chi|_\lambda \in H^1(K_\lambda, \Q/\Z),$$ to be the specialization of $\chi$ to $\rm{Spec}(K_\lambda)$ induced by the canonical composition
$$\rm{Spec}(K_\lambda)\rw \rm{Spec}(\co_{X,\lambda}^h)\rw X.$$ The log Swan conductor of $\chi|_\lambda$ is denoted as $$\rm{Sw}_\lambda(\chi)_{\rm{log}}.$$
\end{defn}
\begin{remark} Since $k$ is perfect and $X/k$ is of finite type, the residue field of $K_\lambda$ has transcendence degree equal to $\dim(X)-1$, and therefore this residue field is perfect if and only if $\dim(X)=1$.\end{remark}
\begin{defn}\label{curveson} A \emph{curve} $C$ is a 1-dimensional integral scheme and a \emph{curve} $C$ \emph{on} a scheme is a closed subscheme that is a curve.\end{defn}
\begin{defn}\label{curvesboundaryE} We let $Z_1(X,E)$ denote the set of curves $\bar{C}$ on $X$ such that $\bar{C}$ is not contained in $E$, i.e. such that $\bar{C}\not\subset \rm{Supp}(E).$
\end{defn}
\begin{defn}\label{infiniteplaces} For $\bar{C}\in Z_1(X,E)$ let $$\Psi_{\bar{C}}: \bar{C}^N\rw \bar{C},$$ denote the normalization of $\bar{C}$. We put $$\bar{C}_\infty = \{z\in |\bar{C}^N| : \Psi_{\bar{C}}(z)\in E\}.$$
\end{defn}
When we want to specify a divisor $D'\subset\rm{Supp}(E)$, we write $Z_1(X,D')$, resp. $\bar{C}_\infty(D')$.
We can think of $\bar{C}_\infty$ as the set of places of the global field $K(\bar{C}) = K(\bar{C}^N)$ ``on the boundary'' $E$. Given $\bar{C}\in Z_1(X,E)$ and $z\in \bar{C}_\infty$, we write $K(\bar{C})_z$ for the henselization of $K(\bar{C})$ at the place corresponding to $z$.
\begin{defn}\label{condcurves} Given $\bar{C}\in Z_1(X,E)$ and $z\in \bar{C}_\infty$, we define $$\chi|_{\bar{C},z}\in H^1(K(\bar{C})_z, \Q/\Z),$$ to be the restriction of $\chi$ via $$\rm{Spec}(K(\bar{C})_z)\rw \bar{C}\hookrightarrow X.$$ Since $\dim{\bar{C}}=1$, we do not need to emphasize a `log' conductor and we unambiguously write $$\rm{Sw}_z(\chi|_{\bar{C}}),$$ for the conductor of $\chi|_{\bar{C},z},$ which coincides with the classical Swan conductor (see Remark \ref{newschool}).
More generally, suppose $\tilde{C}$ is a normal curve with a given finite morphism $$\phi:\tilde{C}\rw X,$$ and that $\tilde{z}\in |\tilde{C}|$ is a closed point satisfying $\phi(\tilde{z})\in E$. We emphasize the morphism $\phi$ by writing $$\rm{Sw}_z(\phi^*\chi),$$ for the conductor of $\chi$ restricted to $\tilde{C}$.
\end{defn}
\subsection{Example of computing by tangent curves}\label{surfacecurves}
Here we provide evidence for Conjecture B in the rank 1 case on a surface. The curves constructed in this example will serve as a prototype for the general higher dimensional case. The motivated reader may gain intuition by plotting some of the following geometric objects.
\begin{exa} (Artin-Schreier for a surface.) Let $k$ be a perfect field of $\rm{char}(k) = p>0 , X = \rm{Spec}(k[x,y]),$ and $E = D = V(y)$ so that $U =
\rm{Spec}(k[x,y][1/y]).$ We will
consider both ``fierce'' and ``non-fierce'' wild ramification (cf.
\cite [\S 2]{lau82}) to give evidence that curves tangent to $D$ handle
both phenomena.
Let $A$ be the henselization of $k[x,y]$ at the prime ideal $(y)$, set $K = \rm{Frac}(A)$, and denote by $v_y$ the normalized discrete valuation on $K$.
Consider the Artin-Schreier equation over $A$: $$t^p - t = x^a/y^b,$$ with $a\geq
0, b\geq 1$. We can further assume that $x^a/y^b$ is not a
$p^{\text{th}}$-power in $K$, i.e. that $p$ is prime to $a$ or $b$.
Let $$L = K[t]/(t^p -
t-x^a/y^b),$$ and let $\tilde{A}$ be the integral closure of $A$ in $L$; then $\tilde{A}$
is also a discrete valuation ring. Let $\bold{k}(A) = k(x)$ denote the
residue field of $A$ and $\bold{k}(\tilde{A})$ the residue field of $\tilde{A}$. The extension $\tilde{A}/A$ is fierce if the extension of residue fields
$\bold{k}(\tilde{A})/\bold{k}(A)$ is inseparable, i.e. if $\bold{k}(\tilde{A})$
contains a $p$-th root of $x$. In the case we consider here, $\tilde{A}/A$ is therefore fierce if $p$ divides $b$ but not $a$.

We consider the curves $C_e$ on $X$
defined by the equation $y = x^e$, for $e>0$. Geometrically, as $e$ increases, the
$C_e$ progressively become more tangent to the $x$-axis in the
plane near the origin. The extension $L/K$ corresponds to an element $\chi\in H^1(K,\Z/p).$ The restriction of $\chi$ to $C_e$ yields a character $\chi |_{C_e}\in H^1(K(C_e), \Z/p)$ that corresponds
to the Artin-Scrheier extension of $K(C_e)$ given by $$t^p - t = x^{a}/x^{be} = \frac{1}{x^{be-a}}.$$
Let $\phi: C_e\hookrightarrow X$ denote the closed immersion and fix a closed point $z\in |C_e|$ such that
$\phi(z)\in D$. The following calculations of the limsup clearly indicate the necessity of considering curves that are not transversal to the divisor $D$.
\begin{itemize}
\item[(i)] Non-fierce case, i.e. $p\nmid b$.
\end{itemize}
The extension of
residue fields $\bold{k}(\tilde{A})/\bold{k}(A)$ is trivial because $(b,p)=1$. Since $v_y(x^a/y^b) = -b$, Example \ref{addexample} gives
$$\rm{Sw}_{(y)}(\chi)_{\log} = b.$$
The Artin-Schreier equation over $C_e$ in this case is $$t^p -t =
\frac{1}{x^{be-a}}.$$ If $be-a\leq 0$, then $\rm{Sw}_{z}(\chi |_{C_e})$ = 0, so we may assume $e>0$ satisfies $be-a>0$.
By \eqref{computeswan} we have that if $p|(be-a)$, then $\rm{Sw}_{z}(\chi |_{C_e}) <be-a$. On the other hand, if $(be-a, p)=1$, then $\rm{Sw}_{z}(\chi |_{C_e}) =
be-a$. Therefore, $$\limsup_{e>0} \frac{\rm{Sw}_{z}(\chi |_{C_e})}{e} =
\lim_{e\rw +\infty} b - \frac{a}{e} = b,$$ as expected. Observe that it was necessary to consider the limsup over\emph{all} $e>0$ in order to ensure we range over those $e>0$ with $(be-a, p)=1$.
\begin{itemize}
\item[(ii)] Fierce case. In this case assume $p|b$ and $a\geq 1$ with $(a,p)=1$.
\end{itemize}
The extension of residue fields is purely inseparable,\footnote{ In fact, let $v_{\tilde{A}}: L^\times\rw \Q$ be a valuation on $L$ normalized by $v_{\tilde{A}}(y)=1$. Then $v_{\tilde{A}}(ty^b) = -1 + b >0$, which implies $x\in \bold{k}(\tilde{A})^p$.}
so we cannot appeal to the classical Swan conductor. We explicitly compute
the conductor by the Brylinski-Kato filtration. Represent $x^a/y^b$ in $W_1(K)/(F-1)$ as $x^a/y^b + f^p - f$, where $f\in K$. Either $v_y(f) < -b/p$ in which case $v_y(x^a/y^b + f^p - f) < -b$, or $v_y(f)\geq -b/p$ in which case $v_y(x^a/y^b + f^p - f) = v_y(x^a/y^b) = -b$ because $x^a$ is not a $p$-th root in $K$. We obtain
$$\rm{Sw}_{(y)}(\chi)_{\log} = b.$$
For each $e>0$, restricting to $C_e$ gives $$t^p -t = x^a/x^{be} =
\frac{1}{x^{be-a}},$$ and since $(be-a,p)=1$, we have
$$\limsup_{e>0} \frac{\rm{Sw}_{z}(\chi |_{C_e})}{e} = \lim_{e\rw+\infty} b - \frac{1}{e} = b,$$ as desired.
\begin{flushright} $\square$ \end{flushright}
\end{exa}
\section{Theorem B1} \label{exp1r1}
For the rest of the paper, $k$ is a perfect field of $\rm{char}(k) = p>0$. Recalling the definitions from \S\ref{swanvariety}, we now state the main theorem of this article.
\begin{thm}\label{tangentrank1} Let $X/k$ be a smooth variety and $U\subset X$ an open subvariety such that $D:=X\backslash U$ is an irreducible divisor on $X$. Denote by $\xi$ the generic point of $D$. Let $\chi\in H^1(U, \Q/\Z)$. Then $$\sup_{\bar{C}}
\frac{\rm{Sw}_z(\chi|_{\bar{C}})}{m_z(\bar\phi^*D)} = \rm{Sw}_\xi
(\chi)_\rm{log},$$ where the supremum ranges over all $\bar{C}\in
Z_1(X,E)$ and also over all points $z\in \bar{C}_\infty\cap D^{sm}$ with multiplicity $m_z(\bar\phi^*D)>0.$
\end{thm}
We prove the theorem at a closed point in \S \ref{local1}, then proceed to the general case in \S \ref{global1}.

\;
\subsection{Local setting}\label{local1}
We begin by formulating the problem around a closed point; the main reason we localize and complete at a closed point is to apply Cohen's structure theorem which enables us to make Artin-Schreier-Witt extensions explicit.
Let $x\in |X|\cap \rm{Supp}(D)$ that is smooth at $D$. Write the completion of $X$ at $x$ as
$$X_x: = \rm{Spec} (\widehat{\co}_{X,x}).$$ Then
$\widehat{\co}_{X,x}$ is a Noetherian and complete regular local ring and by Cohen's structure theorem (see e.g. \cite[4.2.28]{liu02}) there is a finite extension $k'/k$ such that $$
\widehat{\co}_{X,x}\cong k'[[t_1,\ldots, t_d]]. $$
Restricting $D$ to $X_x$, we assume
without loss of generality that locally around $x$, it is
defined by $(t_1 = 0)$. Let $U = X_x\backslash V(t_1)$. We prove
the following
\begin{thm}\label{localtangent} (Local Tangent Theorem) Let $\chi\in H^1(U) = H^1(U,\Q/\Z).$ There are irreducible, \emph{regular} curves $\phi_e: C_e\rw X_x$ with $C_e\in Z_1(X_x, D)$ such that
\begin{equation}\limsup_{e\rightarrow\infty} \frac{\rm{Sw}_x
(\chi|_{C_e})}{m_x(\phi_e^*D)} = \rm{Sw}_{\xi}(\chi)_{\log}.
\label{limit}\end{equation}
\end{thm}
\begin{remark} Since
$X_{x}$ is a local scheme, so is $C_e$, and $|C_e| =
\{x\}$.
\end{remark}
Theorem \ref{localtangent} is obvious if $\dim(U) =1$, for then we
simply take $C_e = X_x$, so from now on we assume $\dim(U)>1$.

By Remarks \ref{cyclic} and \ref{tame}, we are reduced
to the case where $\chi\in H^1(U, \Z/p^{n+1})$ for some $n\geq 0$ (recall $\text{char}(k) =p>0$).

We prove the Local Tangent Theorem \ref{localtangent} first for $n=0$ in \S \ref{ascase} and then proceed to the general case in \S \ref{aswproof}.
We fix the following notation. Let \begin{equation} \label{inclusion1} A_x: =k'[[t_1,\ldots, t_d]],\end{equation} and
\begin{equation} \label{inclusion2} R := \Gamma(U, \co_U) = A_x[1/t_1].\end{equation}
\;
Set $$N = \rm{Sw}_{\xi}(\chi)_{\log} .$$ If $N=0$ (i.e. if $\chi$ is tame at $\xi$), then for \emph{any} subscheme $Z\subset X_x$, the restriction $\rm{Sw}_x(\chi|_Z)$ is zero. Hence we assume that $N>0$ in what follows. Furthermore, by the Artin-Schreier-Witt isomorphism (\ref{ASWK}), $\chi$ corresponds to a unique element $$\ol{\mathbf{f}} = \ol{(f_0,\ldots, f_{n})}\in W_{n+1}(R)/(F-1).$$

We will use the following lemma in the sequel when reducing to curves. It expresses a form of upper semi-continuity\footnote{ Recall that on a metric space, a function $f$ is upper semi-continuous near $x_0$ if $\limsup_{x\rw x_0} f(x) \leq f(x_0)$. For $x<x_0$ we think of $f(x)$ as $f(x; C,z) := \rm{Sw}_x(\chi|_C)/m_z(\phi^*D)$ and $f(x_0)$ as the number $\rm{Sw}_\xi (\chi)_\rm{log}$. } of the conductor's restriction to curves.
\begin{lem}\label{allcurves}
If $\phi:C_x\rw X_x$ is a curve on $X_x,$ then we have the inequality
$$\rm{Sw}_{x}(\chi|_{C_x})\leq m_x(\phi^*D) \,\rm{Sw}_\xi(\chi)_{\rm{log}}.$$
\end{lem}
\begin{proof}
This assertion is clear on pulling-back Witt vectors over discrete valuation rings from $\co_{X_x,\xi}^h$ to $\co_{C_x,x}^h$. 

\end{proof}

\begin{remark} For related results, see \cite[Prop 2.7 (i)]{kersai} and \cite[Coro 2.8]{kersai}.
\end{remark}

\subsubsection{The Artin-Schreier case}\label{ascase}
In this case, $\chi$ corresponds to a unique element $\bar{f}\in R/(F-1)$. We choose a lift $f\in R$ satisfying $$-v_{t_1}(f)= N.$$ Note that $f\not\in R^p$. Write\begin{equation} \label{repff} f= \frac{B}{t_1^N}+\frac{B'}{t_1^{N-1}},\end{equation}
where $B$ is a non-zero element of $k'[[t_2,\ldots, t_d]]$ and $B'\in A_x$. \begin{remark}\label{modifyf} If $p|N$, we modify $f$ so that each term of $B$ is not a $p$th power in $A_x$ as follows. Suppose that for our initial choice of $f$, the term $B$ has elements of the form $h^p$ for some (non-zero) $h\in A_x$. Write $N = pN'$ for some $0<N'<N$. Then in $R/(F-1)$ we have $$\frac{h^p}{t_1^{pN'}} \equiv \frac{h}{t_1^{N'}},$$ and $$\frac{B}{t_1^N} \equiv \frac{B-h^p}{t_1^N}+\frac{h}{t_1^{N'}} .$$ Therefore, we may choose $f\in R$ so that $h$ is a term of $B'$ in the expression $f = B/t_1^N + B'/t^{N-1}$. With this choice of $f$, each term in $B$ is not a $p$th power in $A_x$.
\end{remark}
\subsubsection*{Defining the curves $C_e$}\label{constructcurvesI}
We construct curves $C_e$ on $X_x$ so that the image of $B$ in
$\Gamma(C_e, \co_{C_e})$ is non-zero. The subscript $e$ is motivated by the fact that we want the multiplicity of $C_e$ with the divisor $V(t_1)$ to be $e$.
We begin by defining a surjective morphism $$\Phi_e: k'[[t_1,\ldots, t_d]] \rw k'[[w]],$$ that sends each
parameter $t_i$ to either a power of $w$ or zero, and satisfies
$\Phi_e(B) \neq 0$.
Write $B$ as a sum of monomials $$B = \sum_{(i_2, \ldots, i_d)}
u_{(i_2, \ldots, i_d)}t_2^{i_2}\cdots t_d^{i_d},$$ where each $i_k\geq 0$ and  $u_{(i_2,
\ldots, i_d)}\in k'$. From this sum, choose a \emph{non-zero}
multi-index $(j_2, \ldots, j_d)$ that has a \emph{minimal} number of
$j_q\neq 0$. By a change of coordinates, we can assume such a term
is a polynomial in $t_2,\ldots, t_r$, for some $r$ with $2\leq
r\leq d$. Denote by $\bb{B}_r$ the set of monomials of $B$ of the
form $t_2^{a_2}\cdots t_r^{a_r}$ with each $a_i>0 \; (1\leq i\leq
r)$.
\begin{lem}\label{mexists}
For $\mathfrak{n} = (n_2, \ldots, n_r)\in \mathbb{N}^{r-1}$, define
$$\Psi^\mathfrak{n}: \mathbb{B}_r\rw \N,$$ by
$$\Psi^{\mathfrak{n}}(t_2^{a_2}\cdots t_r^{a_r}) = n_2a_2+\cdots +
n_ra_r.$$
Let $g = t_2^{b_2}\cdots t_r^{b_r}\in \mathbb{B}_r$ be the monomial of minimal degree with respect to the
lexicographical ordering\,\footnote{
$(m,n) \leq (m',n') \Lr m < m'$ or ($m = m'$ and $n \leq n'$).}.
Then there is an element $\gm = (m_2,\ldots, m_r)\in \mathbb{N}^{r-1}$ (depending on the minimal term $g\in \mathbb{B}_r$) satisfying
\begin{equation}\label{mstrict} m_r=1 \; \text{and} \; \Psi^\mathbf{\gm}(g) < \Psi^\mathbf{\gm}(h),\forall h\in \mathbb{B}_r\backslash \{g\}.\end{equation}
\end{lem}
The method employed in the following construction of $\gm$ is very similar in spirit to the classical proof of
Noether normalization (see e.g. \cite[Prop. 2.1.9]{liu02}).
\begin{proof}
Write $\mathbf{b} = (b_2,\ldots, b_r)$.
Define
$$\mathbf{m} = (m_2, \ldots, m_r),$$
via
$$m_r= 1, \; m_j =\sum_{i=j+1}^r m_ib_i, \, (2\leq j<r).$$
Now assume $\Psi^{\mathbf{m}}(h)\leq\Psi^{\mathbf{m}}(g)$ for some $h = t_2^{a_2}\cdots t_r^{a_r} \in
\bb{B}_r$. We claim that $h=g$. First suppose $r=2$, so $m_2=1$. Since $g$ is of minimal degree,
we have $a_2\geq b_2$. Moreover, the hypothesis $\Psi^{\mathbf{m}}(h)\leq\Psi^{\mathbf{m}}(g)$ in this case means $a_2\leq b_2$; therefore, $a_2=b_2$ and the claim
is true for $r=2$. Now we prove the claim for $r\geq 3$. We have the following inequalities:
\begin{equation}
\label{sstar} \sum_{i=2}^r m_i a_i \leq \sum_{i=2}^r m_i b_i
\end{equation}
\begin{equation}
a_2 + \frac{\sum_{i=3}^r m_i a_i}{m_2} \leq b_2 + 1\label{nebula}
\end{equation}
\begin{equation}
(a_2-b_2) + \frac{\sum_{i=3}^r m_i a_i}{m_2} \leq 1. \label{star}
\end{equation}
The inequality \eqref{nebula} is obtained by dividing by $m_2$ and
using the definition $m_2 = \sum_{i=3}^r m_i b_i.$
Since $\mathbf{b}$ is minimal in the lexicographical ordering, we
have $a_2\geq b_2$, and so the inequality \eqref{star} implies that
either
\begin{itemize}
\item[(i)] $a_2 = b_2+1$ and $\sum_{i=3}^r m_i a_i = 0$, or
\item[(ii)] $a_2 = b_2$ and $\sum_{i=3}^r m_i a_i = \sum_{i=3}^r m_i b_i.$
\end{itemize}
Case (i) is impossible since the choice of $\mathbf{b}$ and
induction on $r$ gives $\sum_{i=3}^r m_i a_i>0$. Only case (ii) is
possible. In this case induction on $r$ gives $a_i=b_i.$ Therefore $h=g$.
\end{proof}
\begin{defn}\label{Bgood} We call an element $\gm\in \mathbb{N}^{r-1}$ satisfying (\ref{mstrict}) in Lemma \ref{mexists} a $B$-\emph{good vector}. A $B$-good vector always has last coordinate $m_r=1$. \end{defn}
\begin{prop}\label{constructcurves} For each $e>0$, there are regular curves
$\phi_e: C_e\hookrightarrow X_x$ with $\phi_e^\#(B)\neq 0$ and such that the multiplicity of $C_e$ with $V(t_1)$ is $e$.
\label{phie}
\end{prop}
\begin{proof}\label{algebraicrelations}
Let $\gm = (m_2,\ldots, m_r)$ be a $B$-good vector. Given $e>0$, our curves
$C_e$ are defined by the $k'$-morphism $\Phi_e:
k'[[t_1,\ldots, t_d]]\rw k'[[w]]$
given by \begin{eqnarray*}
\Phi_e(t_1) &=& w^e\\
\Phi_e(t_2) &=& w^{m_2}\\
\Phi_e(t_3) &=& w^{m_3}\\
&\vdots& \\
\Phi_e(t_r) &=& w^{m_r}\\
\Phi_e(t_s) &=& 0, \, (s>r).
\end{eqnarray*}
Since $\gm$ is a $B$-good vector, we have $\Phi_e(g)\neq 0$ and thus $\Phi_e(B)\neq 0.$ The morphism
$$\phi_e:C_e \rw X_x,$$ is defined as the morphism of affine schemes that corresponds to $\Phi_e$.
Note that in terms of ideals (recall $m_r=1$), for $$I =
(t_1-t_r^e, t_2-t_r^{m_2}, \ldots, t_{r-1}-t_r^{m_{r-1}}, \,
t_{r+1}, \ldots, t_d)\leq A_x,$$ we have the equality of closed
subschemes $$C_e = V(I)\subset X_x.$$ \end{proof}
By our construction above, a $B$-good vector $\gm$ is independent of the multiplicity $e>0$. Hence, the image $\Phi_e(B)$ is independent of $e$ since this image depends only on $\gm$. Only the image $t_1$ under $\Phi_e$ depends on $e$. In order to keep track of both the multiplicity $e$ and a given $B$-good vector $\gm$, from now on we write $\Phi_{e,\gm}$.
The following lemmas will be used in the sequel.
\begin{lem}\label{addqis}
Suppose $r>2$. Again write $g = t_2^{b_2}\cdots t_r^{b_r}$ for the element of minimal degree in $\mathbb{B}_r$ with respect to lexicographic order. Suppose that $\gm = (m_2,\ldots, m_r=1)$ is a $B$-good vector. Then if $Q_i \in \mathbb{N} \; (i = 2,\ldots, r-1)$ are integers satisfying
\begin{equation}\label{qis} Q_i \geq \sum_{j=i+1}^{r-1} b_j.Q_j, \end{equation}
then $$\gm' = (m_2+Q_2,\ldots, m_{r-1}+Q_{r-1}, m_r=1),$$ is a $B$-good vector.
\end{lem}
\begin{proof}
We follow exactly the same method as in the proof of Lemma \ref{mexists}. Suppose $h = t_2^{a_2}\cdots t_r^{a_r} \in\mathbb{B}_r$ satisfies
$$\Psi^{\gm'}(h)\leq \Psi^{\gm'}(g).$$ For ease of notation, let $Q_r = 0$. We have
$$(a_2-b_2) + \sum_{i=3}^r(m_i a_i + Q_i a_i)\leq \sum_{i=3}^r (m_i b_i + Q_i b_i).$$ The minimality hypothesis on $g$ and induction on $r$ gives $a_i = b_i$, for $i=2,\ldots, r$. Hence, $\gm' = (m_2+Q_2, \ldots, m_{r-1} + Q_{r-1}, m_r=1)$ is a $B$-good vector.
\end{proof}
If $p|N$, we modify $\gm$ as follows. Denote by $v$ the normalized discrete valuation on $k'[[w]].$
\begin{lem}\label{pdivN} Suppose that $p|N$ and that $\gm$ is a $B$-good vector. There is a $B$-good vector $\gm'$ satisfying $$p\nmid v(\Phi_{e,\gm'}(B)).$$
\end{lem}
\begin{proof}
Recall that by Remark \ref{modifyf}, each term of $B$ is not a $p$th power. Also recall that $g = t_2^{b_2}\cdots t_r^{b_r}\in \mathbb{B}_r$ denotes the minimal lexicographic element of $\mathbb{B}_r$. Now, if $p\nmid v(\Phi_{e,\gm}(B))$ or $r=2$, no modification is necessary and we put $\gm' = \gm$. So now assume $r>2$ and $p|v(\Phi_{e,\gm}(B))$.
We may choose the minimal element $\ell\in [2,r]$ such that $p\nmid b_\ell.$ By definition of a $B$-good vector we will require the last coordinate of $\gm'$ to be $m'_r = m_r = 1$, so we consider two cases. First, suppose $\ell < r$. Set $Q_\ell =1, Q_i = 0$ for $i>\ell$, and $Q_i$ satisfying \eqref{qis}  for $i<\ell$. Define $\gm'$ by $m'_i = m_i + Q_i$ for $i= 2,\dots, r.$ By Lemma \ref{addqis}, $\gm'$ is a $B$-good vector and $$v(\Phi_{e,\gm'}(B)) \equiv b_\ell \not\equiv 0 \mod p.$$

Finally, suppose $\ell=r$. Then $p | b_i$ for $i<\ell$, $p\nmid b_{\ell}$, and $m_{\ell}=1$.

The proof is complete.
\end{proof}
Note that in Lemmas \ref{addqis} and \ref{pdivN}, the modification $\gm'$ of $\gm$ is independent of $e.$ This observation will be used implicitly in the calculations of the conductor restricted to curves $C_e$ below in \S\ref{limsup1}.
\subsubsection{The limsup in the Artin-Schreier case}\label{limsup1}
Suppose we are given $e>0$ and a $B$-good vector $\gm$. If $p|N$, modify $\gm$ as in Lemma \ref{pdivN}. We now turn to calculating the Swan conductor over the curves $\phi_e: C_e\rw X_x$
defined by the morphism of rings $\Phi_{e,\gm}.$ Let $$\bb{K} = K(C_e)\cong
\rm{Frac}\left(k'[[w]]\right),$$ be the function field of $C_e$.
Then $\bb{K}$ is a complete discrete valuation field and we write
$v$ for the normalized discrete valuation on $\bb{K}$.
Recall that $R = A_x[1/t_1]$ and denote again by
$$\Phi_{e,\gm}: R\rw \bb{K},$$ the corresponding morphism of rings. 
Write $$c := v(\Phi_{e,\gm}(B)).$$ Since $-v(\Phi_{e,\gm}(B/t_1^N)) > -v(\Phi_{e,\gm}(B'/t_1^{N'})$ for all $e\gg 0$, the term $B/t_1^N$ of $f$ determines the non-zero rational number $$\frac{\rm{Sw}_x(\chi|_{C_e})}{m_x(\phi_e^*D)} =\frac{\rm{Sw}_x(\chi|_{C_e})}{e} .$$ We can further assume that $e>0$
is sufficiently large that $Ne-c >0$ (recall that $\gm$ and $N>1,$ are independent of $e$; hence,
$c\geq 0$ is independent of $e$).
We now calculate $\rm{Sw}_x(\chi|_{C_e})$ using the Brylinski-Kato
filtration (\ref{aswfil}) on $W_1(\bb{K})/(F-1)$. The image by $\Phi_{e,\gm}$
of $B/t_1^N\in R$ is, up to a unit, equal to $w^{-(Ne-c)}$ in
$\bb{K}$. There are two cases to consider.
\begin{itemize}
\item[(i)] The first case is $p|(Ne-c)$. Then, $w^{-(Ne-c)} =
w^{-p^a.N'}$ for some $a>1$ and $N'\geq 1$ with $(N',p)=1$. Then,
$w^{-(Ne-c)}$ is equal to $w^{-N'}$ in $\bb{K}/(F-1)$. In this
case, $\rm{Sw}_x(\chi|_{C_e}) < Ne-c.$
\item[(ii)] The second case is $p\nmid (Ne-c)$. Using arguments similar to those in the fierce case of Example \ref{surfacecurves}, we have $\rm{Sw}_x(\chi|_{C_e}) = -v(w^{-(Ne-c)}) = Ne-c.$
\end{itemize}
Clearly, if $p\nmid N$, there are infinitely many $e>0$ such that case (ii) above holds. If $p|N$, then as mentioned above, we modify $\phi_e$ by Lemma \ref{pdivN} to achieve case (ii) above for infinitely many $e>0.$
Therefore, to verify the asserted limsup in \eqref{limit} it is sufficient to take the limit over those $e>0$ in case (ii) and we have
$$\limsup_{e\rw\infty}\frac{\rm{Sw}_x(\chi|_{C_e})}{m_x(\phi_e^*D)}
= \lim_{e\rw\infty}\frac{Ne-c}{e} = N,$$ as desired. The proof of Theorem \ref{localtangent} for $\chi\in H^1(U, \Z/p)$ is complete. \begin{flushright} $\square$ \end{flushright}
\subsubsection{The general case: Artin-Schreier-Witt}\label{aswproof}
Given $n\geq 1$, we prove that Theorem \ref{localtangent} is true for cyclic $\Z/p^{n+1}$-extensions, assuming it is true for all cyclic $\Z/p^{n'}$-extensions with $n'<n+1$. The proof ends in \S \ref{endlocaltangentproof}.

Fix a character $\chi\in H^1(U, \Z/p^{n+1})$ and put $$N := \rm{Sw}_\xi(\chi)>0.$$Recall that $R = \Gamma(U,\co_U) = A_x[1/t_1].$ Suppose $\mathbf{\ol{f}} \in W_{n+1}(R)/(F-1)$ corresponds to $\chi$. Choose elements $f_i\in R$, for $i=0,\ldots, n$, such that on writing $$\mathbf{f} = (f_0,\ldots, f_n)\in W_{n+1}(R),$$ 
\begin{itemize}\label{threeconds}
\item[(i)] $\mathbf{f} \equiv \mathbf{\ol{f}} \mod (F-1)W_{n+1}(R),$
\item[(ii)] there is a $k\in [0,n]$ such that
\begin{equation}\label{maxf} -p^{n-k}v(f_k) = N,\end{equation}
\item[(iii)] the minimal $k$ for which \eqref{maxf} holds is maximal among all other $\mathbf{f}'\in W_{n+1}(R)$ that satisfy (i) and (ii).
\end{itemize}
The existence of such an $\bold{f}$ follows directly from \eqref{computeswan}. We proceed by analyzing the case where $k<n$ and then with $k=n$.
First, suppose $k<n$. We apply the induction hypothesis as follows. Let $K = \rm{Frac}(R)$. From the natural embedding $R \rw K$, we have $\chi\in H^1(K,\Z/p^{n+1}).$
Denote by $V^{k+1}: W_{n-k}(K)\rw W_{n+1}(K)$ the $(k+1)$-iterated Verschiebung. The diagram \eqref{Vercommute} extends to a commutative diagram
\begin{equation}\label{Vercommute2}\xymatrix{
0 \ar[d] && 0 \ar[d]\\
W_{n-k}(K) \ar[d]^{V^{k+1}} \ar[rr]^{\delta_{n-k}}&& H^1(K, \Z /p^{n-k}) \ar[d]^{\cdot p^{k+1}}\\
W_{n+1}(K) \ar[d]^{\text{mod} \, V^{k+1} }\ar[rr]^{\delta_{n+1}} && H^1(K, \Z /p^{n+1}) \ar[d]^{\text{mod} \, p^{k+1}}\\
W_{k+1}(K) \ar[rr]^{\delta_{k+1}} \ar[d] && H^1(K, \Z/p^{k+1}) \ar[d] \\
0 && 0 } \end{equation}
Denote by $\Psi^{k+1}$ the morphism corresponding to $\text{mod} \, p^{k+1}$ in the above diagram, so $$\Psi^{k+1}: H^1(K, \Z/p^{n+1}) \rw H^1(K, \Z/p^{k+1}).$$ Let $\chi' = \Psi^{k+1}(\chi)$ and suppose $(z_0,\ldots, z_n)\in W_{n+1}(K)/(F-1)$ corresponds to $\chi$. Then $\chi'$ corresponds to $(z_0,\ldots, z_k)\in W_{k+1}(K)/(F-1)$.
Now, by (\ref{Vercommute2}) the conductor of $\chi$ is determined by the $k+1$ coordinates $z_0,\ldots, z_k$. Since $\mathbf{f}$ satisfies the above three conditions (i) - (iii) we have
$$\rm{Sw}_\xi (\chi)_{\rm{log}} = p^{n-k}\rm{Sw}_\xi(\chi')_{\rm{log}}.$$ 

\begin{lem} Keep $\mathbf{f}$ as above. For any $B$-good vector $\mathfrak{n}$ and $e>0$, write $\Phi_{\mathfrak{n}, e}(\mathbf{f}) = (\Phi_{\mathfrak{n}, e}(f_0), \ldots, \Phi_{\mathfrak{n}, e}(f_n))$. Let $C$ be the curve corresponding to $\Phi_{\mathfrak{n}, e}$ via Prop \ref{constructcurves}. Then $\rm{Sw}(\chi|_C)\leq Ne.$
\end{lem}
\begin{proof} This follows by our choice of $\mathbf{f}$ and \eqref{computeswan} above.
\end{proof}

Therefore, the commutativity of the diagram (\ref{Vercommute2}) and our induction hypothesis applied to $\chi'$ implies the assertion of Theorem \ref{localtangent} in the case $k<n$.

It remains to analyze the case $k=n$, i.e.
$$N = -v(f_n) \; \text{and} -v(f_n)>-p^{n-i}v(f_i), \forall i\in [0,n-1].$$ We first construct the morphism $\Phi_{e,\gm} : R\rw \bb{K}$ and then proceed to calculate $\limsup_{e>0} \rm{Sw}_x(\chi|_{C_e})/m_x(\phi_e^*D)$, where $\phi_e: C_e\rw X_x$ is the morphism of schemes defined by $\Phi_{e,\gm}$.
For each $f_i\neq 0$, let $$N_i= -v(f_i),$$ with $N=N_n$ and as in \eqref{repff} write
\begin{equation} \label{repf} f_i= \frac{B_i}{t_1^{N_i}} + \frac{B'_i}{t_1^{N_i -1}}, \text{with} \; 0\neq B_i\in k'[[t_2, \ldots, t_n]] \; \text{and} \; B_i'\in A_x.\end{equation}
In order to define $\Phi_{e,\gm}$, consider the term $B_n$ of $f_n$ and fix an integer $e>0$. Then we define $\Phi_{e,\gm} : R\rw \bb{K}$ with respect to $B_n$ by setting $B = B_n$ so that $\Phi_{e,\gm}$ is defined by a $B$-good vector as in Prop. \ref{constructcurves}. Then, $$\Phi_{e,\gm}(B_n)\neq 0, \forall e>0.$$ 

Set
$$\Phi_{e,\gm}(\mathbf{f})= (\Phi_{e,\gm}(f_0), \ldots, \Phi_{e,\gm}(f_n))\in W_{n+1}(\bb{K}).$$
For each $e\gg 0$ we claim that $\Phi_{e,\gm}(f_n)$ determines
the conductor corresponding to the image of $\Phi_{e,\gm}(\mathbf{f})$ in $H^1(\bb{K}, \Z/p^{n+1})$ upon dividing by the multiplicity $m_x(\phi_e^*D) = e$.

Recall that 
$$\delta_{n+1}: W_{n+1}(\bb{K}) \rw H^1(\bb{K}, \Z/p^{n+1}),$$  is the connecting homomorphsim of groups in \eqref{Vercommute} and for $m \geq 0$,
$$\rm{fil}_{m}H^1(\bb{K}, \Z/p^{n+1}) := \delta_{n+1}(\rm{fil}_{m} W_{n+1}(\bb{K})).$$


\begin{lem}\label{limitfilmod} Keep $\mathbf{f}\in W_{n+1}(R)$ as above. For $e>0$, consider the composite $$\ol{\Phi_{e,\gm}} : W_{n+1}(R)\stackrel{\Phi_{e,\gm}}{\longrightarrow} W_{n+1}(\bb{K}) \stackrel{\delta_{n+1}}{\longrightarrow} H^1(\bb{K}, \Z/p^{n+1}).$$
For an integer $h\geq 1$, let $s(h)\geq 0$ denote the \emph{minimal} integer for which 
$$\ol{\Phi_{h,\gm}} (\mathbf{f})\in \rm{fil}_{s(h)}H^1(\bb{K}, \Z/p^{n+1}).$$ 

Then $$\limsup_{e\rw\infty} s(e)/e = \sup_{e'} s(e')/e' = N,$$ where $e'$ ranges over all $e'>0$ for which $p\nmid \,v(\Phi_{e', \gm} (f_n)).$
\end{lem}

\begin{proof}
We have $-p^{n-i}v(f_i) < -v(f_n) = N$ and for each $z\in (F-1)R$, $v(f_n + z)\leq v(f_n)$. Therefore, there is a $B$-good vector $\gm$ satisfying the above conditions with respect to $\mathbf{f}$ and $e\gg 0$ such that 
$$-v(\Phi_{e, \gm}(f_n))= Ne - v(\Phi_{e, \gm}(B_n)),$$ and
$$ -p^{n-i}v(\Phi_{e, \gm}(f_i)) \leq p^{n-i}N_i < Ne - v(\Phi_{e, \gm}(B_n)).$$
For example, we may take $ e > v(\Phi_{e, \gm}(B_n))$. Furthermore, choose $e$ and $\gm$ such that $Ne - v(\Phi_{e,\gm}(B_n))$ is not divisible by $p$. 

Now $\Phi_{e, \gm}(\mathbf{f})$ satisfies the conditions of Lemma \ref{gamma}, thus 
$$\limsup_{e\rw\infty} \frac{s(e)}{e} = \limsup_{e\rw \infty} \frac{Ne - v(\Phi_{e, \gm}(B_n))}{e} \geq N.$$

By Lemma \ref{allcurves}, $\limsup_{e\rw\infty} s(e)/e \leq N$ which gives $\limsup_{e\rw\infty} s(e)/e  = N.$

Recalling $p\nmid Ne - v(\Phi_{e,\gm}(B_n))$ we also have $\sup_{e'} s(e')/e' = N.$

\end{proof}

\subsubsection{End of proof of Theorem \ref{localtangent}}\label{endlocaltangentproof}
By induction we have reduced to the case where $\rm{Sw}_\xi(\chi)_\rm{log} = -v(f_n).$ Lemma \ref{limitfilmod} enables us to use the technique in \S\ref{limsup1} to compute $N:=\rm{Sw}_\xi(\chi)_{\rm{log}}$ in terms of the $C_e$. Recall that the curves $\phi_e: C_e\rw X_x$ on $X_x$ are defined by the term $B_n$ of the coordinate $f_n$ in $\mathbf{\ol{f}}$. Let 
$$c= v(\Phi_{e,\gm}(B_n)).$$
 By Lemma \ref{limitfilmod} 
$$\limsup_{e\rw\infty} \frac{\rm{Sw}(\chi|_{C_e})}{m_x(\phi_e^*D)} = \lim_{e\rw\infty} \frac{Ne - c}{e} = N.$$

This completes the proof of the Local Tangent Theorem, Thm. \ref{localtangent}.
\begin{flushright} $\square$ \end{flushright}

\subsection{Global setting}\label{global1} We proceed to
the global setting, i.e. to Thm. \ref{tangentrank1}, via the following two propositions.
\begin{prop} \label{localswan}
Suppose $D$ is a smooth irreducible divisor on $X$ with generic point
$\xi$, and let $x\in |X|\cap \rm{Supp}(D)$. Since $D$ and $X$ are regular, there is a unique point $\lambda\in \rm{Spec}\, \co_{X,x}^h$ lying over $\xi\in \rm{Spec} \,\co_{X,x}$. Consider the henselian discrete valuation rings $A =
\co_{X, \xi}^h$ and $B= ({(\co_{X,x})^h}_{\lambda})^h$. Denote by $\omega$ the generic point of $\rm{Spec}\, B.$ Let $\chi\in H^1(U,\Q/\Z)$. Then \begin{equation}\label{localswansequal} \rm{Sw}_\xi
(\chi)_{\rm{log}} = \rm{Sw}_\lambda
(\chi|_{\omega})_{\rm{log}}.\end{equation}
\end{prop}
\begin{proof}
This is clear since $B/A$ is unramified (note that the extension of residue fields is separably generated).
\end{proof}
In order to ensure that we get 1-dimensional subschemes on $X$ from the local setting above, we use henselian local rings of $X$, as opposed to complete local rings, since they are unions of finitely-generated $k$-algebras.
Recall that the curves $C_e$ constructed in the proof of Prop. \ref{constructcurves} are defined by algebraic relations in the parameters of the complete local ring $\widehat{\cal{O}}_{X,x} =: A_x$. Moreover, denote the subset of power series in $\widehat{\cal{O}}_{X,x}$ that satisfy an algebraic relation over $\co_{X,x}$ by $(\widehat{\co}_{X,x})^{\text{alg}}$. Then $(\widehat{\co}_{X,x})^{\text{alg}}$ is a subset of the henselization $\co_{X,x}^h =: A_x ^h$.\footnote{The converse is true by Artin's Approximation theorem, hence in $\widehat{\co}_{X,x}$ we have the equality $(\widehat{\co}_{X,x})^{\text{alg}} = \co_{X,x}^h$. We do not require this fact.} 

Choose from the beginning local parameters $t_1, \ldots, t_n$ of $\co_{X,x}$. Then the curves $C_e$ are defined by algebraic relations in these $t_1, \ldots, t_n$. Therefore, we have curves $C_e^h$ on $\rm{Spec}(A_x^h)$ for the following
\begin{prop}\label{hencurves} There are regular curves $\phi_e: C_e^h\rw \rm{Spec}(A_x^h)$ with $C_e^h\in Z_1(\rm{Spec}(A_x^h), D)$ indexed by $e>0$ such that
\begin{equation}\limsup_{e\rightarrow\infty} \frac{\rm{Sw}_x
(\chi|_{C_e^h})}{m_x(\phi_e^*D)} = \rm{Sw}_{\xi}(\chi)_{\log}.
\label{limit-hen}\end{equation}
\begin{flushright} $\square$ \end{flushright}
\end{prop}
We are now in a position to prove Theorem \ref{tangentrank1}.
\subsection{Proof of Theorem \ref{tangentrank1}}
\begin{proof}
Let $x\in |X|\cap \rm{Supp}(D)$. Let $X_x^h = \rm{Spec}(A_x^h)$ and write $$C'_e\subset X_x^h,$$ for the curves
given in Prop. \ref{hencurves} above and consider the image $f_x(C'_e)$, where $f_x:
X_x^h \rw X$ is the canonical morphism. Since $x$ is a
closed point of $X$ and $X_x^h$ is the henselization of
$X$ at $x$, it follows that the \emph{closure} of
$f_x(C'_e)$ is a curve on
$X$. In fact, the Zariski closure of
$f_x(C_{e}')$ has dimension 1 since the
function field of $f_x(C_e ')$ has transcendence degree 1 over $k$. Then the $f_x(C_{e}')$'s comprise the desired set of curves $\bar{C}$ on $X$. Recall that by definition of $C_\infty(D)$  in \ref{infiniteplaces}, the conductor $\rm{Sw}_z(\chi|_C)$ is computed on the normalizations $\bar{C}^N$. Since $\dim(\bar{C}) = 1$, $\bar{C}^N$ is a \emph{regular} scheme.
The theorem then follows immediately by combining Prop. \ref{localswan} and the Local Tangent Theorem \ref{localtangent}.
\end{proof}
\begin{coro}\label{allcurvesglobal} For each curve $\bar\phi: \bar{C}\rw X$ on $X$ and $z\in \bar{C}_\infty(D)$ with $\bar\phi(z)$ a smooth point of $D$,
$$\rm{Sw}_z(\bar\phi^*\chi)\leq m_z(\bar\phi^*D)\, \rm{Sw}_\xi
(\chi)_\rm{log}.$$
\end{coro}
The proof follows by combining Lemma \ref{allcurves} and Prop. \ref{localswan}.
\section{Theorem A1}
Conjecture A in our setting is the following theorem. Recall that for an effective Cartier divisor $D$, $Z_1(X,D)$ denotes the set of curves on $X$ that are not contained in the support of $D$. Given $\bar{C}\in Z_1(X,D)$, $$\bar{C}_\infty(D)= \{z\in |\bar{C}^N|: \Psi_{\bar{C}}(z)\in D\}.$$
\begin{thm}(Conjecture A in the smooth rank 1 case)\label{thmexp2} Let $X/k$ be a smooth variety and $U\subset X$ an open subvariety such that the reduced subscheme $E=X\backslash U$ is the support of an effective Cartier divisor on $X$, and that regarded as a $k$-scheme is smooth. Let $\chi\in H^1(U, \Q/\Z)$. Let $D\in \rm{Div}^+(X)$ be a regular sncd with $\rm{Supp}(D) = E.$ The following are equivalent.
\begin{itemize}
\item[(i)] For each curve $\bar\phi: \bar{C}\rw X$ and each $z\in\bar{C}_\infty(D)$,
$$\rm{Sw}_z(\bar\phi^*\chi) \leq m_z(\bar\phi^*D).$$
\item[(ii)] For each generic
point $\xi \in D$,
$$\rm{Sw}_{\xi
}(\chi)_{\log}\leq m_{\xi }(D).$$
\end{itemize}
\end{thm}
\begin{remark} The simplifying assumption that $\rm{Supp}(D) = E$ is imposed, instead of $\rm{Supp}(D)\subset E$, so the notation $\bar{\phi}^*\chi$ makes sense. \end{remark}
\subsection{Proof of Theorem \ref{thmexp2}}
\begin{proof}
For ease of notation, set $$m_\xi := m_\xi(D).$$
(i)$\Rw$ (ii):
Given $\xi$, let $D'$ be the component of $D$ containing $\xi$ such that in $\rm{Div}^+(X)$ we have $$D' = m_\xi D'_{\rm{red}}.$$
By hypothesis, for each curve $\bar{C}\in Z_1(X,D')$ and each point $z\in \bar{C}_\infty(D')$, we have
\begin{equation}\label{multleq} \rm{Sw}_z(\bar{\phi}^*\chi)\leq m_z(\bar{\phi}^*{ D'}).
\end{equation}
By Thm. \ref{tangentrank1}, we know that
\begin{equation}\label{swleq} \rm{Sw}_\xi(\chi)_{\rm{log}} =
\sup_{z',\bar{C}}\frac{\rm{Sw}_{z'}(\bar{\phi}^*\chi)}{m_{z'}(\bar{\phi}^*( D'_{\rm{red}}))},
\end{equation}
where the supremum ranges over $\bar{C}\in Z_1(X,D'^{\rm{sm}})$ and points $z'\in \bar{C}_\infty(D'^{\rm{sm}})$ satisfying the property that the image of $z'$ in $X$ is contained in $D'$ but not any other component of $D$.
For such points $z'$, we further have
$$m_{z'}(\bar{\phi}^* D') \leq m_{z'}(\bar{\phi}^* (D'_{\rm{red}})) m_\xi.$$
By \eqref{multleq} we have $$\rm{Sw}_{z'}(\bar{\phi}^*\chi)\leq m_{z'}(\bar{\phi}^* (D'_{\rm{red}})) m_\xi,$$ and then
$$\frac{\rm{Sw}_{z'}(\bar{\phi}^*\chi)}{m_{z'}(\bar{\phi}^*( D'_{\rm{red}}))} \leq m_\xi.$$
Thus, by \eqref{swleq} we have
$$\rm{Sw}_\xi(\chi)_{\rm{log}}\leq m_\xi.$$
\;
(ii)$\Rw$ (i): We may assume $\chi\in H^1(U, \Z/p^n)$ for some $n\geq 1.$
First suppose that $D$ is irreducible with generic point $\lambda$.
By Coro.\,\ref{allcurvesglobal}, we may assume that $m_z(\bar\phi^*D)>0$. There is an integer $M_{\bar{C}}\geq 1$ such that $$m_z(\bar\phi^*D) = M_{\bar{C}} \, m_{\lambda}(D).$$ The proof of Thm. \ref{localtangent} and Coro. \ref{allcurvesglobal} yields that $$\rm{Sw}_z(\bar{\phi}^*\chi) \leq M_{\bar{C}} \, \rm{Sw}_{\lambda}(\chi)_\rm{log}.$$ By hypothesis, $\rm{Sw}_{\lambda}(\chi)_\rm{log}\leq m_{\lambda}(D)$ and so $$\rm{Sw}_z(\bar{\phi}^*\chi) \leq M_{\bar{C}} \, m_{\lambda}(D) =m_z(\bar\phi^*D),$$ as desired.
These arguments and purity of the branch locus complete the proof for the case where the sncd $D$ has more than one component.
\end{proof}
\section{Class groups with modulus and Existence Theorem}\label{modulus}
Recently, M. Kerz and S. Saito proved an Existence Theorem in higher class field theory by using, among other things, bounded extensions and conductors. They employ a so-called `non-log' filtration that is similar to the `log' filtration in rank 1, the latter being the Brylinski-Kato filtration, see Def. \ref{aswfil}. We now discuss the relationships between the log and non-log filtrations that allows us to deduce the log version of their Existence Theorem.
We recall the non-log filtration employed in \cite{kersai}, which was introduced by Matsuda in \cite{mat}. Let $K$ be a henselian discrete valuation field of characteristic $p>0$. Recall the Brylinski-Kato filtration $\rm{fil}_m$ on $W_n(K)$ in Def. \ref{aswfil} and set
\begin{equation}\label{nonlogdef} \rm{fil}^{\text{nonlog}}_m W_n(K) := \rm{fil}_{m} W_n(K) + V^{n-n'}\rm{fil}_m W_{n'}(K),\end{equation}
where $n' = \min\{n, \rm{ord}_p(m)\}.$
As in Def.\eqref{primezero} from \S\ref{ASWdiscussion}, \eqref{nonlogdef} yields a so-called \emph{non-log} filtration on all of $H^1(K)$, which we write as $\rm{fil}^{\rm{nonlog}}_m$. This filtration satisfies the following properties.
\begin{lem} \label{Florian}
The following properties are satisfied by the non-log filtration.
\begin{itemize}
\item[(i)] $\rm{fil}^{\rm{nonlog}}_1 H^1(K)$ is the subgroup of tamely ramified characters.
\item[(ii)] $\rm{fil}^{\rm{nonlog}}_m H^1(K)\subset \rm{fil}_m H^1(K)\subset \rm{fil}^{\rm{nonlog}}_{m+1} H^1(K)$.
\item[(iii)] $\rm{fil}^{\rm{nonlog}}_m H^1(K) = \rm{fil}_{m-1}H^1(K)$ if $(m,p)=1$.
\end{itemize}
\end{lem}
For simplicity, let $X/k$ a normal variety and $U/k$ an open subvariety such that the reduced closed complement $X\backslash U$ is an sncd $D$ on $X$.
Consider the following property on a character $\chi\in H^1(U):= H^1(U, \Q/\Z) = \rm{Hom}_\rm{cont}(\pi_1^{\rm{ab}}(U),\Q/\Z)$ : for each generic point $\xi\in D$,
\begin{equation}\label{swbound} \rm{Sw}_\xi(\chi)_{\rm{log}}\leq m_\xi(D). \end{equation}
\begin{defn}\label{rayclass} For $D$ as above, define $\rm{fil}_D H^1(U)$ to be the subgroup of $H^1(U)$ consisting of those $\chi$ satisfying (\ref{swbound}) above. Also define
$$\pi_1^{\rm{ab}}(X,D)_\rm{log} = \rm{Hom}_\rm{cont}(\rm{fil}_D H^1(U), \Q/\Z),$$ with the pro-finite topology of the dual.
\end{defn}
The topological group $\pi_1^{\rm{ab}}(X,D)_\rm{log}$ is a quotient of $\pi_1^{\rm{ab}}(U)$ and should be thought of as classifying extensions of $U$ with ``ramification bounded by $D$''.
Following \cite[\S 3.4]{ru}, let $Z_0(X,D)$ be the subgroup of $0$-cycles on $X\backslash D$, and set
$$\cal{R}(X,D) = \{(C,f): C\in Z_1(X,E), f\in K(C)^\times, \,\tilde{f} \equiv 1 \,\text{mod} \, D^\# \},$$
where $\tilde{f}$ is the image of $f$ in $K(C^N)$ and $D^\# = (D-D_{\rm{red}})\cdot C^N + (D\cdot C^N)_{red}.$
Let $R_0(U,D)$ be the subgroup of $Z_0(X, D)$ generated by $\rm{div}(f)_C$ for $(C,f)\in \cal{R}(X,D).$ Then we set $$\rm{C}(X, D)_\rm{log} = Z_0(X, D)/R_0(X, D).$$
Now consider the degree zero parts
$$\rm{C}(X,D)_{\rm{\log}}^0 = \rm{ker}[\rm{C}(X,D)_{\rm{log}}\stack{\rm{deg}}{\longrightarrow}\Z],$$ and
$$\pi^{\rm{ab}}_1(X,D)_\rm{log}^0 = \rm{ker}[\pi^{\rm{ab}}_1(X,D)_{\rm{log}}\rw \pi_1^{\rm{ab}}(\rm{Spec}(k))].$$

We now reformulate the non-log Existence Theorem of \cite{kersai} to the log case. We replace the Artin conductor used in \cite{kersai} with our Swan conductor and their relative class group $\rm{C}(X,D)$ with the group $\rm{C}(X,D)_{\text{log}}.$ Moreover, the results applied from Matsuda's work on restriction to curves \cite[Coro 2.8]{kersai} is replaced by Theorem \ref{thmexp2}. Finally, in the course of the proof of \cite[Coro IV]{kersai}, properness plays an essential role in reducing to the 2-dimensional case, e.g. in applying Serre duality. Thus, we also need to add the properness hypothesis to use their arguments. Finally, following the same steps as outlined in the overview of their proof as given in \cite[pg. 15]{kersai}, we obtain
\begin{thm}(Existence Theorem, with log filtration) Suppose $k$ is a finite field with $\rm{char}(k)\neq 2$. Let $X/k$ be a \emph{proper} \emph{and} smooth $k$-variety and $U\subset X$ an open $k$-subscheme such that the reduced closed complement $D=X\backslash U$ is an sncd on $X$. Then $$\rm{C}(X, D)^0_\rm{log}\rw
\pi_1^{\rm{ab}}(X,D)_{\rm{log}}^0,$$ is an isomorphism of \emph{finite} abelian groups.
\end{thm}
\begin{flushright} $\square$ \end{flushright}
\section{Acknowledgements}
I want to thank M. Kerz for suggesting the problem of restricting to curves and discussions on ramification theory. I also thank S. Kelly, T. Saito, and L. Xiao for further beneficial discussions on the state-of-the-art of this theory. I would also like to thank the referees for a close reading and their suggestions that greatly improved the exposition. Further gratitude is fondly extended to M. Borkowski, S. Kelly, and J. Ross for several helpful comments which improved the clarity of this article.
The Emmy Noether foundation provided generous scholarship support during part of the development of this work.


\begin{thebibliography}{99}
\bibitem[Abb-Sai 02]{abbsai02} A. Abbes and T. Saito, \textit{Ramification of local fields
with imperfect residue fields} Amer. J. Math., 124(5): p. 879 - 920, 2002.
\;
\bibitem[Abb-Sai 09]{abbsai09} A. Abbes, T. Saito, \textit{Analyse micro-locale l-adique en caractéristique $p>0$. Le cas d 'un trait}, Publ. RIMS, Kyoto Univ. 45, p. 25 - 74, 2009.
\;
\bibitem[Abb-Sai 11]{abbsai11} A. Abbes, T. Saito,
\textit{Ramification and cleanliness},
Tohoku Math. J. (2) 63, (2011), no. 4, p. 775 - 853.
\;
\bibitem[And 07]{And07} Y. Andr\'e, \textit{An algebraic proof of Deligne's regularity criterion.} Algebraic, analytic and geometric aspects of complex differential equations and their deformations. Painleve hierarchies, RIMS Kokyuroku Bessatsu, B2, Res. Inst. Math. Sci. (RIMS), Kyoto, 2007, p. 1 - 13.
\;
\bibitem[Bry 83]{bry83} J.-L. Brylinski, \textit{Th\'eorie du corps de classes de Kato et
rev\^etements ab\'eliens de surfaces}, Ann. Inst. Fourier 33 (1983),
p. 23 - 38.
\;
\bibitem[Chi-Pul 09]{chia}
B. Chiarellotto,  A. Pulita, \textit{Arithmetic and differential Swan conductors of rank one
representations with finite local monodromy}, Amer. J. Math 131 (2009), no. 6, p. 1743-
1794.
\;
\bibitem[Del 76]{del76} P. Deligne, Letter to L. Illusie of 28.11.76, unpublished.
\;
\bibitem[Esn-Ker 12]{esnker12} H. Esnault and M. Kerz, \textit{A finiteness theorem for
Galois representations of function fields over finite fields (after
Deligne)}, Acta Mathematica Vietnamica 37, Number 4 (2012), p. 531 - 562.
\;
\bibitem[Kat 89]{Kat89} K. Kato, \textit{Swan conductors for characters of degree one in the imperfect residue field case}, Algebraic K-theory and algebraic number
theory (Honolulu, HI, 1987), p. 101 - 131, Contemp. Math., 83, Amer. Math.
Soc., Providence, RI, 1989.
\;
\bibitem[Ker-Sai 15]{kersai} M. Kerz, S. Saito, \textit{Chow group of 0-cycles with modulus and higher dimensional class field theory}, preprint, arXiv:1304.4400v4.
\;
\bibitem[Ker-Sch 10]{kerscht} M. Kerz, A. Schmidt, \textit{On different notions of tameness in arithmetic geometry}, Math. Annalen 346, Issue 3 (2010), p. 641 - 668.
\;
\bibitem[Lau 81]{lau81} G. Laumon \textit{Semi-continuit\'e du conducteur de Swan
(d'apr\`es P. Deligne). Caract\'eristique d'Euler-Poincar\'e}, p.
173 - 219, Ast\'erisque, 82 - 83, Soc. Math. France, Paris, 1981.
\;
\bibitem[Lau 82]{lau82} G. Laumon, \textit{Caract\'eristique d'Euler-Poincar\'e des
faisceaux constructibles sur une surface}, Ast\'erisques, 101 - 102
(1982), p. 193 - 207.
\;
\bibitem[Liu 02]{liu02} Q. Liu. \textit{Algebraic geometry and arithmetic curves}, volume 6 of Oxford Graduate Texts in Mathematics. Oxford University Press, Oxford,
2002.
\;
\bibitem[Mat 97]{mat} S. Matsuda, \textit{On the Swan conductors in positive characteristic}, Amer. J. Math., 119, (1997), p. 705 - 739.
\;
\bibitem[Ru]{ru} H. Russell, \textit{Albanese varieties with modulus over a perfect field}, Algebra \& Number Theory 7, no. 4 (2013), p. 853 - 892.
\;
\bibitem[Sai 09]{sai09} T. Saito, \textit{Wild ramification and the characteristic cycle of an l-adic sheaf}, Journal de l?Institut de Mathematiques de Jussieu, (2009) 8(4), p. 769 - 829.
\;
\bibitem[Sai 10]{sai10} T. Saito, \textit{Wild ramification of schemes and sheaves},
Proceedings of the international congress of mathematicians (ICM
2010), Hyderabad.
\;
\bibitem[Sai 13]{sai13} T. Saito, \textit{Wild ramification and the cotangent bundle}, preprint, arXiv:1301.4632.
\;
\bibitem[Ser 68]{ser1} J.-P. Serre, \textit{Corps Locaux}, Hermann, Paris 1968.
\;
\bibitem[Wie 06]{wie1} G. Wiesend, \textit{A construction of covers of arithmetic schemes}, J. Number Theory 121 (2006), no. 1, p. 118 - 131.
\;
\bibitem[Xia-Zhu 13]{xiaozhukov} L. Xiao, I. Zhukov, \textit{Ramification of higher local fields, approaches and questions}, Proceedings of The
Second Valuation Theory Conference. Available at www.math.usask.ca/fvk/Xiao-Zhukov.pdf
\;
\bibitem[Zhu 08]{zhu08} I. Zhukov, \textit{Ramification of surfaces: sufficient jet order for wild jumps}, preprint, arXiv:0201071.
\end{thebibliography}
\end{document}